\documentclass[12pt,reqno]{article}

\usepackage[usenames]{color}
\usepackage{graphicx}
\usepackage{amscd}
\usepackage[colorlinks=true,
linkcolor=webgreen,
filecolor=webbrown,
citecolor=webgreen]{hyperref}
\usepackage{cite}

\definecolor{webgreen}{rgb}{0,.5,0}
\definecolor{webbrown}{rgb}{.6,0,0}
\usepackage{amsthm}
\usepackage{fullpage}
\usepackage{enumerate}
\usepackage{graphics,amsmath,amssymb}
\usepackage{amsfonts}
\usepackage{latexsym}
\usepackage{url}
\newcommand{\seqnum}[1]{\href{http://oeis.org/#1}{\underline{#1}}}

\newcommand\BD{\mathrm{B}}
\newcommand\SD{\mathrm{S}}

\theoremstyle{plain}
\newtheorem{theorem}{Theorem}
\newtheorem{corollary}[theorem]{Corollary}
\newtheorem{lemma}[theorem]{Lemma}
\newtheorem{proposition}[theorem]{Proposition}

\theoremstyle{remark}

\begin{document}

\begin{center}
\vskip 1cm{\LARGE\bf Star of David and other patterns in the Hosoya-like polynomials triangles}
\vskip 1cm
\large
Rigoberto Fl\'orez\\
Department of Mathematics and Computer Science\\
The Citadel\\
Charleston, SC \\
U.S.A.\\
{\tt rigo.florez@citadel.edu} \\
\ \\
Robinson A. Higuita\\
Instituto de Matem\'aticas\\
Universidad de Antioquia\\
Medell\'in\\
Colombia\\
{\tt robinson.higuita@udea.edu.co}\\
\ \\
Antara Mukherjee\\
Department of Mathematics and Computer Science\\
The Citadel\\
Charleston, SC \\
U.S.A.\\
{\tt antara.mukherjee@citadel.edu}

\end{center}

\vskip .2 in

\begin{abstract}

In this paper we first generalize the numerical recurrence relation given by Hosoya to polynomials. Using this generalization we construct a
Hosoya-like triangle for polynomials, where its entries are products of generalized Fibonacci polynomials (GFP).
Examples of GFP are: Fibonacci polynomials, Chebyshev polynomials, Morgan-Voyce polynomials, Lucas polynomials, Pell polynomials,
Fermat polynomials, Jacobsthal polynomials, Vieta polynomials and other familiar sequences of polynomials. For every choice of a GFP we obtain a
triangular array of polynomials. In this paper we extend  the star of David property, also called the Hoggatt-Hansell identity, to this type of triangles.
We also establish the star  of David property in the gibonomial triangle.
In addition, we study other geometric patterns in these triangles and as a consequence we give geometric interpretations for the Cassini's identity, Catalan's identity, and other identities for Fibonacci polynomials.
\end{abstract}

\section {Introduction}

The \emph{generalized Fibonacci polynomial} (GFP) is a recursive polynomial sequence that generalizes
the Fibonacci numbers sequence. Familiar examples of GFP are Fibonacci polynomials, Chebyshev polynomials,
Morgan-Voyce polynomials, Lucas polynomials, Pell polynomials, Fermat polynomials, Jacobsthal polynomials, Vieta polynomials,
and other familiar sequences of polynomials. Most of the polynomials mentioned here may be found in \cite{koshy, koshy1}.

The \emph{Hosoya triangle}, formerly called the Fibonacci triangle,  \cite{florezHiguitaMuk, florezjunes, hosoya, koshy}, consists of a
triangular array of numbers where each entry is a product of two Fibonacci numbers (see \seqnum{A058071}).
In this triangle if we replace Fibonacci numbers with the corresponding GFP, we obtain the Hosoya like polynomial triangles (see Tables \ref{tabla1} and \ref{tabla_equivalent}).
For brevity we call these triangles the Hosoya polynomial triangles and if there is any ambiguity we call them Hosoya triangles.
 Therefore, for every choice of GFP we obtain a distinct Hosoya polynomial triangle.
So, every polynomial evaluation gives rise to a  numerical triangle (see Table \ref{Tableahosoyatriangles}).
In particular the classic Hosoya triangle can be  obtained by evaluating the entries of Hosoya polynomial
triangle at $x=1$ when they are Fibonacci polynomials.

The Hosoya polynomial triangle provides a good geometry to study algebraic and combinatorial
properties of products of recursive sequences of polynomials. In this paper we study some of its geometric properties.
Note that any geometric property in this triangle is automatically true for the classic (numerical) Hosoya triangle.

A hexagon gives rise to the star of David --connecting its alternating vertices with a
continuous line-- as in Figure \ref{starofDavidF} part (d) on page~\pageref{starofDavidF}.
Given a hexagon in a Hosoya polynomial triangle can we determine whether the vertices of
the two triangles of the star of David have the same greatest common  divisor (GCD)? If
both GCD's are equal, then we say that the star of David has the \emph{GCD property}.
Several authors have been interested in this property, see for example
\cite{florezHiguitaJunes,florezjunes,hoggatt_star_david,koshygibonomial,Korntved,Long, sun}.
For instance, in 2014 Fl\'{o}rez et. al. \cite{florezHiguitaJunesGCD} proved the Star of
David property in the generalized Hosoya triangle. Koshy \cite{koshygibonomial, koshy1} defined the
gibonomial triangle and proved one of the fundamental properties of the star of David in this triangle.
In this paper we establish the GCD property of the star of David for the gibonomial triangle.

Since every polynomial that satisfies the definition of GFP gives rise to a Hosoya polynomial triangle,
the above question seems complicated to answer.  We prove that the star of David property holds for
most of the cases (depending on the locations of its points in the Hosoya polynomial triangle).
We also prove that if the star of David does not hold, then the two GCD's are proportional.
We give a characterization of the members of the family of Hosoya polynomial triangles that satisfy the
star of David property. From Table \ref{equivalent}, we obtain a sub-family of fourteen distinct Hosoya
polynomial triangles. We provide a complete classification of the members that satisfy the star of David property.

We also study other geometric properties that hold in a Hosoya polynomial triangle, called the
\emph{rectangle property} and the \emph{zigzag property}. A rectangle in the  Hosoya polynomial triangle is a
set of four points in the triangle that are arranged as the vertices of a rectangle. Using the
rectangle property we give geometric interpretations and proofs of the Cassini, Catalan, and
Johnson identities for GFP.

 \section{Generalized Fibonacci polynomials GFP}\label{General:Fibonacci:Polynomial}

In this section we summarize the definition of the generalized Fibonacci polynomial given by the authors in an earlier article, \cite{florezHiguitaMukCharact}.
The \emph{generalized Fibonacci polynomial} sequence, denoted by GFP, is defined by the following recurrence relation
\begin{equation}\label{Fibonacci;general}
G_0(x)=p_0(x), \; G_1(x)= p_1(x),\;  \text{and} \;  G_{n}(x)= d(x) G_{n - 1}(x) + g(x) G_{n - 2}(x) \text{ for } n\ge 2
\end{equation}
where $p_0(x)$ is a constant and $p_1(x)$, $d(x)$, and $g(x)$ are non-zero polynomials in $\mathbb{Z}[x]$ with $\gcd(d(x), g(x))=1$.
Some familiar examples of GFP are in Table \ref{equivalent} (also see \cite{florezHiguitaMukCharact, florezHiguitaMuk, Pell,Fermat,koshy}).

A sequence given by (\ref{Fibonacci;general}) is called \emph{Lucas type} or \emph{first type} if $2p_{1}(x)=p_{0}(x)d(x)$ with  $p_{0}\ne 0$, and a sequence given by
(\ref{Fibonacci;general}) is called \emph{Fibonacci type} or \emph{second type} if $p_{0}(x)=0$ with  $p_{1}(x)$ a constant, however in this paper we consider $p_{1}(x)$ to be $1$.
We use the notation $G_n^{*}(x)$ when $G_n(x)$ is of Lucas type and $G_n^{'}$ when $G_n(x)$ is of Fibonacci type.
Using these definitions of Lucas type and Fibonacci type polynomials the authors \cite{florezHiguitaMukCharact} found closed formulas for the
GFP that are similar to Binet formulas for the classical numerical sequences like Fibonacci and Lucas numbers.

If $d^2(x)+4g(x)> 0$, then the explicit formula for the recurrence relation (\ref{Fibonacci;general}) is given by
\begin{equation}\label{solutionrecurrencerelationuno}
 G_{n}(x) = t_1 a^{n}(x) + t_2 b^{n}(x)
\end{equation}
where $a(x)$ and $b(x)$ are the solutions of the quadratic equation associated to the second order
recurrence relation $G_{n}(x)$. That is,  $a(x)$ and $b(x)$ are the solutions of $z^2-d(x)z-g(x)=0$
(for details on the construction of Binet formulas see \cite{florezHiguitaMukCharact}).
So, the Binet formula for the GFP of Lucas type is
\begin{equation}\label{bineformulados}
L_n(x)=\frac{a^{n}(x)+b^{n}(x)}{\alpha}
\end{equation}
where $\alpha=2/p_{0}(x)$. The Binet formula for the GFP of Fibonacci type when $p_1(x)=1$ is
\begin{equation}\label{bineformulauno}
R_n(x)=\frac{a^{n}(x)-b^{n}(x)}{a(x)-b(x)}.
\end{equation}
Note that $a(x)+b(x)=d(x)$, $a(x)b(x)= -g(x)$, and $a(x)-b(x)=\sqrt{d^2(x)+4g(x)}$ where
$d(x)$ and $g(x)$ are the polynomials defined in \eqref{Fibonacci;general}.
For the sake of simplicity, throughout this paper we use $a$ in place of $a(x)$ and $b$ in place of $b(x)$.

A GFP sequence of Lucas (Fibonacci) type is \emph{equivalent} or \emph{conjugate} to a sequence of the Fibonacci (Lucas) type,
if their recursive sequences are determined by the same polynomials $d(x)$ and $g(x)$. Notice that two equivalent polynomials
have the same $a(x)$ and $b(x)$ in their Binet representations. Examples of equivalent polynomials are in Table \ref{equivalent}.
Note that the leftmost polynomials in Table \ref{equivalent} are of the Lucas type and their equivalent Fibonacci type polynomials
are in the third column on the same line.

 For most of the proofs involving GFP of Lucas type it is required that
 $\gcd(p_0(x), p_1(x))=1$, $\gcd(p_0(x), d(x))=1$, $\gcd(p_0(x), g(x))=1$, and $\gcd(d(x), g(x))=1$.
Therefore, for the rest the paper we suppose that these four mentioned conditions hold for all GFP.
We use $\rho$ to denote $\gcd(d(x),G_1(x))$. Notice that in the definition Pell-Lucas we have that
$p_0(x)=2$ and $p_1(x)=2x$. Thus, the $\gcd(p_0(x),p_1(x))\ne 1$.
Therefore, Pell-Lucas does not satisfy the extra conditions that we just imposed for Generalized Fibonacci polynomial.
To solve this inconsistency we define \emph{Pell-Lucas-prime} as follows:
\[Q_0'(x)= 1, \; Q_1'(x)= x, \; \text{and} \;  Q_{n}^{'}(x)= 2x Q_{n - 1}^{'}(x) + Q_{n - 2}^{'}(x) \text{ for } n\ge 2.\]
It is easy to see  that $2Q_{n}^{'}(x)=Q_{n}(x)$. Fl\'orez, Junes, and Higuita  \cite{florezHiguitaJunes}, have worked on similar problems for
numerical sequences.

\begin{table} [h]
\begin{center}\scalebox{0.8}{
\begin{tabular}{|lc|lc|l|l|} \hline
   Polynomial  	    & $L_n(x)$    	     &Polynomial of 	&$R_n(x)$   &$a(x)$ 	              	& $b(x)$			\\	
   Lucas type 	 	& 				     & Fibonacci type  	&			& 					        &      				\\ \hline \hline
   Lucas 			&$D_n(x)$     		 &Fibonacci 		&$F_n(x)$ 	&  $(x+\sqrt{x^2+4})/2$     & $(x-\sqrt{x^2+4})/2$ \\ 						
   Pell-Lucas 		&$Q_n(x)$	 	     &Pell				& $P_n(x)$  &  $x+\sqrt{x^2+1}$	      	& $x-\sqrt{x^2+1}$        \\
   Fermat-Lucas 	& $\vartheta_n(x)$	 & Fermat 			& $\Phi_n(x)$ &  $(3x+\sqrt{9x^2-8})/2$ & $(3x-\sqrt{9x^2-8})/ 2$ \\
   Chebyshev first kind& $T_n(x)$ 		 &Chebyshev second kind&$U_n(x)$ &  $x +\sqrt{x^2-1}$      	& $x -\sqrt{x^2-1}$       \\
   Jacobsthal-Lucas	& $j_n(x)$	   		 &Jacobsthal  		& $J_n(x)$   &  $(1+\sqrt{1+8x})/2$    	& $(1-\sqrt{1+8x})/2$      \\
   Morgan-Voyce 	&$C_n(x)$ 	   	     &Morgan-Voyce	    & $B_n(x)$ 	 &  $(x+2+\sqrt{x^2+4x})/2$ & $(x+2-\sqrt{x^2+4x})/2$  \\
   Vieta-Lucas 		&$v_n(x)$ 	   	     &Vieta	            & $V_n(x)$ 	 &  $(x+\sqrt{x^2-4})/2$    & $(x-\sqrt{x^2-4})/2$     \\   \hline
\end{tabular}}
\end{center}
\caption{$R_n(x)$ equivalent to $L_n(x)$.} \label{equivalent}
\end{table}

\section{Divisibility properties of GFP}

In this section we prove a few divisibility and $\gcd$ properties that are true for all GFP. These
results will be used in a section later on to prove the main results of this paper.
Lemma \ref{prop2;1} is a generalization of \cite[Proposition 2.2]{florezjunes}, both proofs are similar. The reader can therefore update the proof in the
afore-mentioned paper to obtain the proof of Lemma \ref{prop2;1}.

\begin{lemma}\label{prop2;1}
Let $p(x), q(x), r(x),$ and $s(x)$ be polynomials.
\begin{enumerate}[(1)]
  \item If $\gcd(p(x),q(x))=\gcd(r(x),s(x))=1$, then $$\gcd(p(x)q(x),r(x)s(x))=\gcd(p(x),r(x))\gcd(p(x),s(x))\gcd(q(x),r(x))\gcd(q(x),s(x)).$$
  \item If $\gcd(p(x),r(x))=1$ and $\gcd(q(x),s(x))=1$, then
     $$\gcd(p(x)q(x),r(x)s(x))=\gcd(p(x),s(x))\gcd(q(x),r(x)).$$
\end{enumerate}
\end{lemma}

\begin{proposition}\label{modulo:dx} If $\{G_{t}(x)\}$  is a  GFP sequence, then
 \[
 G_m(x) \bmod d^2(x)\equiv
 \begin{cases}
         g^{k-1}(x)\left(kd(x)G_{1}(x)+g(x)G_{0}(x)\right), & \mbox{if $m=2k$;} \\
         g^{k}(x)\left(kd(x)G_{0}(x)+G_{1}(x)\right), & \mbox{if $m=2k+1$.}
\end{cases}
 \]
\end{proposition}

\begin{proof} We use mathematical induction. Let $S(m)$ be the statement
 \[
 G_m(x) \bmod d^2(x)\equiv
 \begin{cases}
         g^{t-1}(x)\left(td(x)G_{1}(x)+g(x)G_{0}(x)\right), & \mbox{if $m=2t$;} \\
         g^{t}(x)\left(td(x)G_{0}(x)+G_{1}(x)\right), & \mbox{if $m=2t+1$.}
\end{cases}
 \]

It is easy to see that
\[ G_{1}(x)\equiv G_{1}(x)=g^{0}(x)\left(0 d(x)G_{0}(x)+G_{1}(x)\right) \bmod d^2(x)\]
and
\[G_{2}(x)\equiv G_{2}(x)=g^{0}(x)\left(d(x)G_{1}(x)+g(x)G_{0}(x)\right) \bmod d^2(x).\]
This proves $S(1)$ and $S(2)$.

We suppose that $S(m)$ is true for $m=2k$ and $m=2k+1$.
The proof of $S(m+1)$ requires two cases, we prove the case for $m+1=2k+2$, the case $m+1=2k+3$ is similar and we omit it.
We know that $G_{m+1}(x)=d(x)G_{m}(x)+g(x)G_{m-1}(x)$. Thus,
$G_{2k+2}(x)=d(x)G_{2k+1}(x)+g(x)G_{2k}(x)$. This and the inductive hypothesis imply that $G_{2k+2}(x) \bmod d^2(x)$ is
\[ d(x) \left[g^{k}(x)\left(kd(x)G_{0}(x)+G_{1}(x)\right)\right]+g(x)\left[g^{k-1}(x)\left(kd(x)G_{1}(x)+g(x)G_{0}(x)\right) \right]. \]
Simplifying we obtain,
 \[G_{2(k+1)}(x) \equiv g^{k}(x)\left((k+1)d(x)G_{1}(x)+g(x)G_{0}(x)\right) \bmod d^2(x).\]
This completes the proof.
\end{proof}

\begin{lemma} [\cite{florezHiguitaMukCharact}] \label{gcddistance1;2} Let $m$ and $n$ be positive integers.
	If $ G_{n}(x)$ is a GFP of either Lucas or Fibonacci type, then

\begin{enumerate}[(1)]	
	 \item $\gcd(d(x), G_{2n+1}(x))=G_1(x)$ for every positive integer $n$.
	
	\item If the GFP is of Lucas type, then $\gcd(d(x), G_{2n}^{*}(x))= 1$ and
	
	if the GFP is of Fibonacci type, then $\gcd(d(x), G_{2n}^{\prime}(x))= d(x)$.
	
	\item $\gcd(g(x), G_n(x))=\gcd(g(x), G_{1}(x))=1,$ for every positive integer $n$.
 \item If $0<|m-n|\le 2$ and  $\{G_{t}^{*}(x)\}$  is a  GFP of Lucas type, then
 \[
 \gcd(G_m^{*}(x),G_n^{*}(x))=
 \begin{cases}
         G_{1}^{*}(x), & \mbox{if $m$ and $n$ are both odd;} \\
         1, & \mbox{otherwise. }
\end{cases}
 \]

   \item If $0<|m-n|\le 2$ and $\{G_{t}^{\prime}(x)\}$  is a  GFP of Fibonacci type, then
 \[
 \gcd(G_m^{\prime}(x),G_n^{\prime}(x))=
 \begin{cases}
        G_{2}^{\prime}(x) & \mbox{if $m$ and $n$ are both even;} \\
         1, & \mbox{otherwise. }
\end{cases}
 \]

\end{enumerate}
\end{lemma}

\section{Hosoya polynomial triangle} \label{HosoyaSection}

We now give a precise definition of both the Hosoya polynomial sequence and the Hosoya polynomial triangle.
Let $\delta(x)$, $\gamma(x)$, $p_{0}(x)$, and $p_{1}(x)$ be polynomials in $\mathbb{Z}[x]$. Then the \emph{Hosoya polynomial sequence} $\left\{H(r,k)\right\}_{r,k\ge 0}$ is defined
using the double recursion
\[ H(r,k)= \delta(x) H(r-1,k)+\gamma(x) H(r-2,k)\]
 and
 \[ H(r,k)= \delta(x) H(r-1,k-1)+\gamma(x) H(r-2,k-2)\]
where $ r>1 $ and $0\le k \le r-1$ with initial conditions
\[H(0,0)=p_0(x)^2; \quad H(1,0)=p_0(x)p_1(x); \quad H(1,1)=p_0(x)p_1(x); \quad H(2,1)=p_1(x)^2.\]
This sequence gives rise to the \emph{Hosoya polynomial triangle}, where the
entry in position $k$ (taken from left to right), of the $r${th} row is equal to $H(r,k)$ (see Table \ref{tabla1}).

\begin{table} [!ht] \small
\begin{center} \addtolength{\tabcolsep}{-3pt} \scalebox{.9}{
\begin{tabular}{ccccccccccc}
&&&&&                                                 $H(0,0)$                                                 &&&&&\\
&&&&                                        $H(1,0)$     &&     $H(1,1)$                                        &&&&\\
&&&                                $H(2,0)$    &&     $H(2,1)$     &&     $H(2,2)$                               &&&\\
&&                        $H(3,0)$   &&     $H(3,1)$     &&     $H(3,2)$      &&    $H(3,3)$                     &&\\
&              $H(4,0)$     &&     $H(4,1)$    &&     $H(4,2)$     &&     $H(4,3)$     &&     $H(4,4)$             &\\
      $H(5,0)$     &&    $H(5,1)$    &&     $H(5,2)$     &&     $H(5,3)$      &&    $H(5,4)$     &&    $H(5,5)$     \\
\end{tabular}}
\end{center}
\caption{Hosoya polynomial triangle.} \label{tabla1}
\end{table}

In this paper we are interested in the relationship between the points of the Hosoya polynomial triangle and the products of generalized Fibonacci polynomials.
Fl\'orez, Higuita, and Mukherjee \cite{florezHiguitaMuk}, proved Proposition \ref{lemma0} below which helped establish the mentioned relation. Thus, from
Proposition \ref{lemma0} we can see that Table \ref{tabla1} is equivalent to Table \ref{tabla_equivalent}. To complete the relation between Hosoya polynomial
triangle and GFP we  need $\delta(x)= d(x)$ and $\gamma(x)=g(x)$ where $d(x)$ and $g(x)$ are the polynomials defined in (\ref{Fibonacci;general})
and $\delta(x)$ and $\gamma(x)$ are the polynomials defined in the Hosoya polynomial sequence. So, for the rest of the paper we assume that
$\delta(x)= d(x)$ and $\gamma(x)=g(x)$. Note  that in Table \ref{tabla_equivalent}, for brevity, we use the notation $G_k$ instead of $G_k(x)$.

\begin{proposition}\label{lemma0} $H(r,k)= G_k(x)G_{r-k}(x)$.
\end{proposition}

The proof of this proposition is similar to the proof  of \cite[Proposition 1]{florezHiguitaJunesGCD} for numerical sequences.

\subsection{A coordinate system for the Hosoya polynomial triangle}\label{coordinatesystem}

If $P$ is a point in a Hosoya polynomial triangle, then it is clear that there are two unique positive integers $r$ and $k$
such that $r > k$ with $P=H(r,k)$. We call the ordered pair $(r,k)$ the \emph{rectangular coordinates} of the point $P$.
Fl\'{o}rez \emph{et al}. \cite{florezHiguitaJunesGCD}, introduced a more convenient system of coordinates for points in
the generalized Hosoya triangle. The mentioned coordinate system generalized naturally to Hosoya polynomial triangle. Thus,
from Proposition \ref{lemma0} it is easy to see that any diagonal of Table \ref{tabla_equivalent} is the collection of
all generalized Fibonacci polynomials multiplied by a particular $G_{n}(x)$. More precisely, an $n$th \emph{diagonal} in
the Hosoya polynomial triangle is the collection of all generalized Fibonacci polynomial multiplied by $G_n(x)$. We
distinguish between \emph{slash diagonals} and \emph{backslash diagonals}, with the obvious meaning. We write
$\SD(G_{n}(x))$ and $\BD(G_{m}(x))$ to mean the slash
diagonal and backslash diagonal, respectively. These two diagonals are

\[\SD(G_{n}(x))=\{ H(n+i,n) \}_{i=0}^{\infty} =\{  G_{n}(x)\,G_{i}(x) | i \in \mathbb{Z}_{\ge 0}\},  \]
and
\[\BD(G_{m}(x))=\{ H(m+i,i) \}_{i=0}^{\infty} = \{ G_{i}(x)\,  G_{m}(x) | i \in \mathbb{Z}_{\ge 0} \}. \]

Using this idea we can now associate an ordered pair of non-negative integers to every element of a
Hosoya polynomial triangle. If $P$ is a point in a Hosoya polynomial triangle, then
there are two polynomials $G_{m}(x)$ and  $G_{n}(x)$ such that $P
\in \BD(G_{m}(x)) \cap \SD(G_{n}(x))$. Thus, $P=G_{m}(x) \, G_{n}(x)$. Therefore, the point $P$ corresponds
to the pair $(m,n)$. It is clear that this correspondence is a bijection between points
of a Hosoya polynomial triangle and ordered pairs of non-negative integers. The pair $(m,n)$
is called the \emph{diagonal coordinates} of $P$. We use Proposition \ref{lemma0} to find the diagonal coordinates of a point $P$ represented in rectangular coordinates. Indeed, the  point $P=H(r,k)$ in rectangular coordinates is $P=(r,k)$. Since $H(r,k)= G_k(x)G_{r-k}(x)$, by Proposition \ref{lemma0} we have that the point $P$ in diagonal coordinates is $P=(k, r-k)$.

\begin{table} [!ht] \small
\begin{center} \addtolength{\tabcolsep}{-1pt} \scalebox{.9}{
\begin{tabular}{ccccccccccc}
&&&&&            $G_0 \, G_0$                             				                                             &&&&&\\
&&&&          $G_0\, G_1$ &&      $G_1 \,G_0$                                           		   	                  &&&&\\
&&&        $G_0\,G_2$   &&     $G_1\,G_1$     &&  $G_2\,G_0$                                                            &&&\\
&&       $G_0\,G_3$   &&     $G_1\,G_2$     &&  $G_2\,G_1$  &&   $G_3\,G_0$                     	                       &&\\
&     $G_0\,G_4$    &&    $G_1\,G_3$     &&   $G_2\,G_2$  &&   $G_3\,G_1$   &&  $G_4\,G_0$                                &\\
    $G_0\,G_5$   &&    $G_1\,G_4$      &&   $G_2\,G_3$   &&   $G_3\,G_2$   &&   $G_4\,G_1$   &&   $G_5\,G_0$                \\
\end{tabular}}
\end{center}
\caption{$H(r,k)= G_k(x)G_{r-k}(x)$.} \label{tabla_equivalent}
\end{table}

Some examples of $H(r,k)$ are in Table \ref{polynomialformula}, obtained from Table \ref{equivalent} using Proposition \ref{lemma0}.
Therefore, some examples of Hosoya Polynomial triangle can be constructed using Tables \ref{tabla_equivalent} and \ref{polynomialformula}.
It is enough to substitute each entry in Table \ref{tabla1} or Table \ref{tabla_equivalent} by the corresponding entry in
Table \ref{polynomialformula}. Thus, we obtain a Hosoya polynomial triangle for each of the specific polynomials mentioned
in Table \ref{equivalent}. So, Table \ref{polynomialformula} gives rise to 14 examples of Hosoya polynomial triangle.

For example, using the first polynomial in  Table \ref{polynomialformula} and Proposition \ref{lemma0} in
Table \ref{tabla_equivalent} we obtain the Hosoya polynomial triangle where the entry $H(r,k)$ is equal to
$F_{k}(x) F_{r-k}(x)$. This is represented in Table \ref{tabla2} without the points that contain the factor $F_{0}(x)=0$.

For Table \ref{polynomialformula} we use  $\delta:=\delta(x)=d(x)$ and $\gamma:=\gamma(x)=g(x)$, the polynomials defined in the
Hosoya polynomial sequence  are referred to as $H(r,k)$, and $p_0:=p_0(x)$ and $p_1:=p_1(x)$ are the polynomials defined in (\ref{Fibonacci;general}).

\begin{table} [!ht]
\begin{center}\scalebox{0.8}{
\begin{tabular}{|l|c|c|c|r||l|c|c|c|r|} \hline
$H(r,k)$ & $p_0$ & $p_1$ & $\delta$ & $\gamma$ & $H(r,k)$& $p_0$ & $p_1$ & $\delta$ & $\gamma$
\\ \hline \hline \noalign {\smallskip}
    $F_k(x)F_{r-k}(x)$      & 0 & $1$  & $x$ & $1$ &
    $D_k(x)D_{r-k}(x)$      & 2 & $2x$ & $2x$ & $1$\\
    $P_k(x)P_{r-k}(x)$      & 0 & $1$ & $2x$ & $1$&
    $Q_k(x)Q_{r-k}(x)$      & 2 & $2x$ & $2x$ & $1$\\
    $\Phi_k(x)\Phi_{r-k}(x)$& 0 & $1$ & $x$ & $-2$&
    $\vartheta_k(x)\vartheta_{r-k}(x)$ & $2$ & $3x$ & $x$ & $-2$\\
    $U_k(x)U_{r-k}(x)$      & 0 & $1$ & $2x$ & $-1$&
    $T_k(x)T_{r-k}(x)$      & 1 & $x$ & $2x$ & $-1$\\
    $J_k(x)J_{r-k}(x)$      & 0 & 1 &1 & $2x$&
    $j_k(x)j_{r-k}(x)$      & 2 & 1 & 1 &$2x$ \\
    $B_k(x)B_{r-k}(x)$      & 0 & $1$ & $x+2$ & $-1$&
    $C_k(x)C_{r-k}(x)$      & 2 & $x+2$ & $x+2$ & $-1$\\
    $V_k(x)V_{r-k}(x)$      & 0 & $1$ & $x$ & $-1$ &
    $v_k(x)v_{r-k}(x)$      & 2 & $x$ & $x$ & $-1$\\
     \hline
\end{tabular}}
\end{center}
\caption{Terms $H(r,k)$ of the Hosoya polynomial triangle.} \label{polynomialformula}
\end{table}

Observe that $H(r,k)$ in the first column of Table \ref{polynomialformula} is a product of polynomials of Fibonacci type. Therefore, $G_0(x)=0$ so the edges containing $G_0(x)$ as a factor in Table \ref{tabla_equivalent}, will have entries equal to zero. From the sixth column of Table \ref{polynomialformula} we see that $H(r,k)$ is a product of polynomials of Lucas type. So the edges containing $G_0(x)$ as a factor in Table \ref{tabla_equivalent} will not have entries equal to zero.

\begin{table} [!ht] \small
\begin{center} \addtolength{\tabcolsep}{-3pt} \scalebox{.9}{
\begin{tabular}{ccccccccccc}
&&&&&                                 $1$                                          &&&&&\\
&&&&                           $x$     &&       $x$                    			    &&&&\\
&&&                $x^2+1$     &&     $x^2$      &&  $x^2+1$             	         &&&\\
&&          $x^3+2x$   &&     $x^3+x$   &&     $x^3+x$  &&     $x^3+2x$               &&\\
& $x^4+3x^2+1$ &&  $x(x^3+2x)$  &&  $(x^2+1)^2$ &&   $x(x^3+2x)$  &&    $x^4+3x^2+1$    &\\
\end{tabular}}
\end{center}
\caption{The Hosoya polynomial triangle where $H(r,k) = F_{k}(x) F_{r-k}(x)$.} \label{tabla2}
\end{table}

\section{Star of David property in the Hosoya polynomial triangle}

In the first part of this section we prove one of the main results of this paper, namely the Star of David property for the
Hosoya polynomial triangle. This property holds in Pascal's triangle, Hosoya triangle, generalized Hosoya triangle, Fibonomial triangle, and gibonomial triangle.

Koshy \cite[Chapters 6 and 26]{koshytriangular} discussed how some properties of star of  David are present in
several triangular arrays. Those properties --called \emph{Hoggatt-Hansell} identity,
\emph{Gould} property, or \emph{GCD} property-- were also proved in \cite{florezHiguitaJunesGCD, florezjunes} for Hosoya
and generalized Hosoya triangles. The results in this paper generalize several results in
\cite{florezHiguitaJunesGCD, florezjunes, hosoya, koshytriangular}
that were proved for numerical sequences.  In particular  in Theorem \ref{gcdstarofdavid} parts (1), (2) and (3)
we prove the \emph{Hoggatt-Hansell} identity  and \emph{Gould} property for polynomials.

Throughout the rest of this paper we use only diagonal coordinates (see Subsection \ref{coordinatesystem})
to refer to any point in a Hosoya polynomial triangle.

In the following part of this section we take, $a_1,a_2, a_3$ and $b_1,b_2,b_3$ as the vertices of the two triangles of
the star of David and $c$ its interior point in the generalized Hosoya polynomial triangle (see Figure \ref{starofDavidF}
parts (a) and (b)). The points  $a_1,a_2, a_3$ and $b_1,b_2,b_3$ can be seen as the alternating points of a hexagon (see Figure \ref{starofDavidF} part (d)).

If we know the location of one vertex, we can obtain the location of the remaining five vertices of the star of David.
For instance, if $(m, n)$ are the diagonal coordinates of $a_{2}$, then  the points in the star of
David in Figure \ref{starofDavidF} part (a) are

\begin{table}[!ht]
\begin{center}
\begin{tabular}{llll}
$a_{1}=G_{m+1}(x)\, G_{n-2}(x),$ \quad &$a_{2}=G_{m}(x)\, G_{n}(x)$, & \text{ and } \quad& $a_{3}=G_{m+2}(x) \, G_{n-1}(x),$ \\
$b_{1}=G_{m}(x)\, G_{n-1}(x),$         &$b_{2}=G_{m+2}(x) \, G_{n-2}(x)$, & \text{ and }  & $b_{3}=G_{m+1}(x) \, G_{n}(x).$
\end{tabular}
\end{center}
\caption{Coordinates for star of David in Figure \ref{starofDavidF} part (a).} \label{coordinates:star:david(a)}
\end{table}
Similarly, if $(m, n)$ are the diagonal coordinates of $b_{2}$, then  the points in the star of
David seen in Figure \ref{starofDavidF} part (b):

\begin{table}[!ht]
\begin{center}
\begin{tabular}{llll}
$a_{1}=G_{m}(x)\, G_{n-1}(x),$         &$a_{2}=G_{m-2}(x) \, G_{n-2}(x)$, & \text{ and }  & $a_{3}=G_{m-1}(x) \, G_{n}(x),$\\
$b_{1}=G_{m-1}(x)\, G_{n-2}(x),$ \quad &$b_{2}=G_{m}(x)\, G_{n}(x)$, & \text{ and } \quad& $b_{3}=G_{m-2}(x) \, G_{n-1}(x).$
\end{tabular}
\end{center}
\caption{Coordinates for star of David in Figure \ref{starofDavidF} part (b).} \label{coordinates:star:david(b)}
\end{table}
Note that the coordinates for the point $c$ in  the star of David in Figure \ref{starofDavidF} part (a)
are given by $G_{m+1}(x)\, G_{n-1}(x)$ and coordinates for the point $c$ in  the star of David in
Figure \ref{starofDavidF} part (b) are given by $G_{m-1}(x)\, G_{n-1}(x)$.

\begin{figure} [!ht]
\begin{center}
\includegraphics[width=33mm]{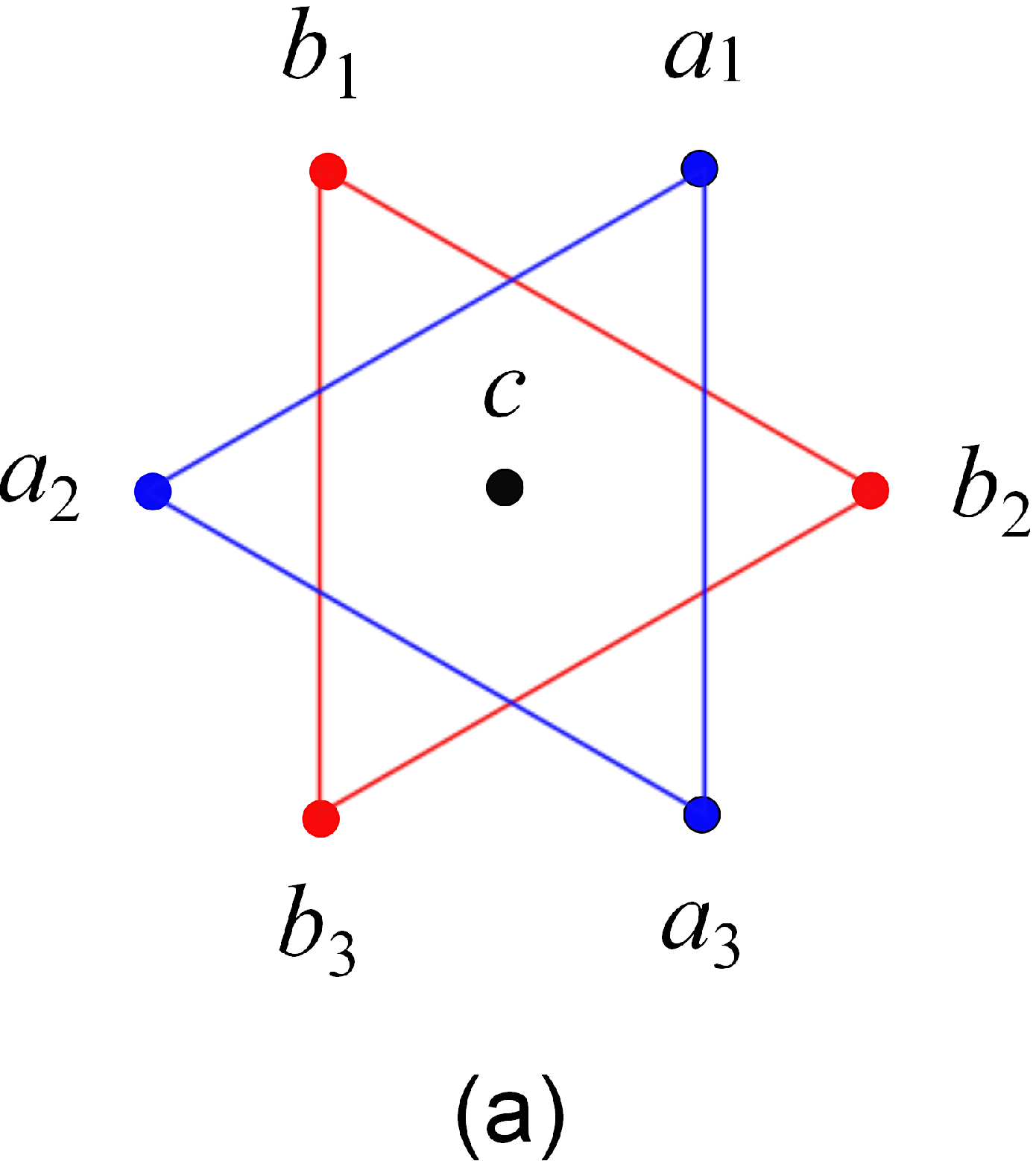} \hspace{.4cm}
\includegraphics[width=33mm]{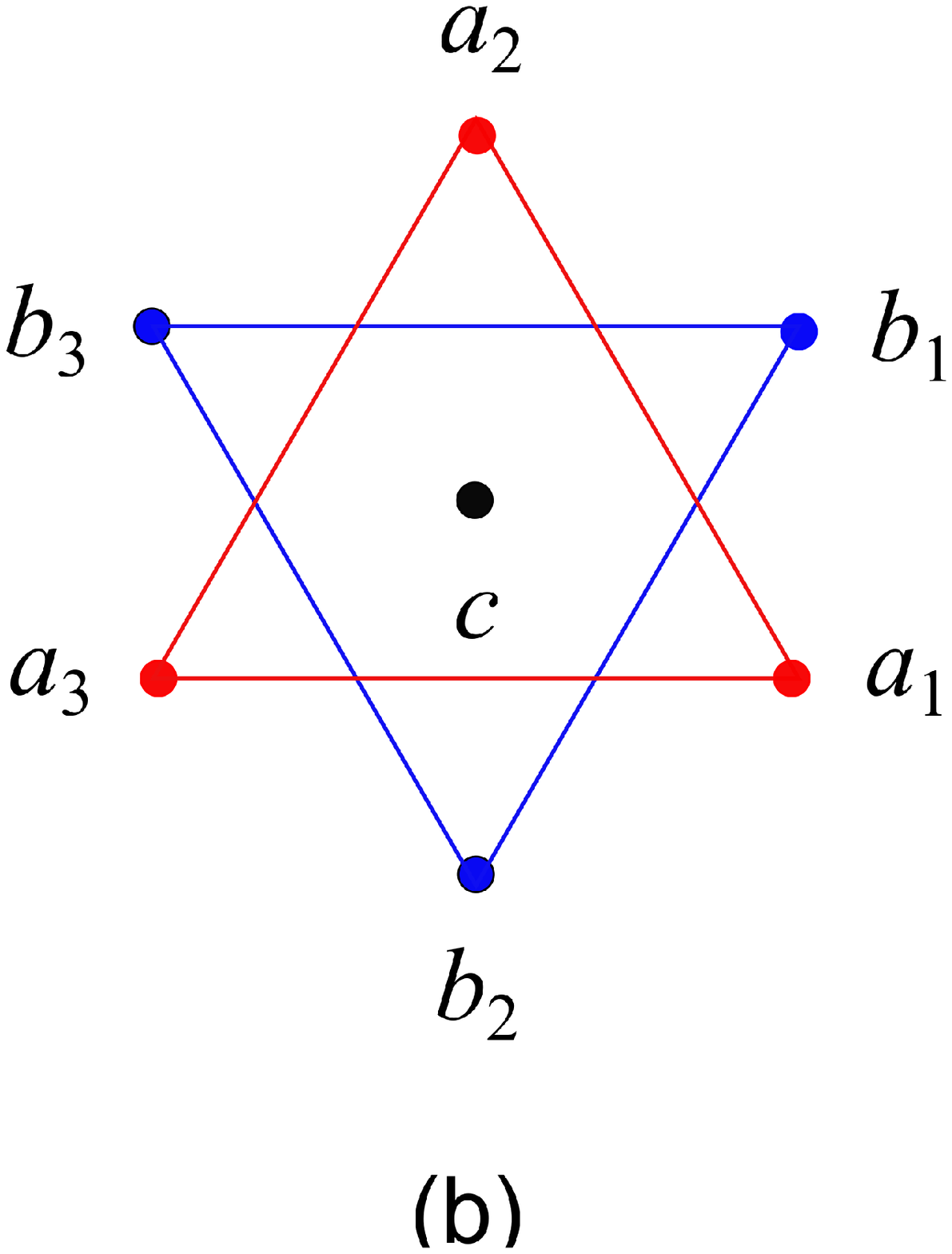} \hspace{.4cm}
\includegraphics[width=33mm]{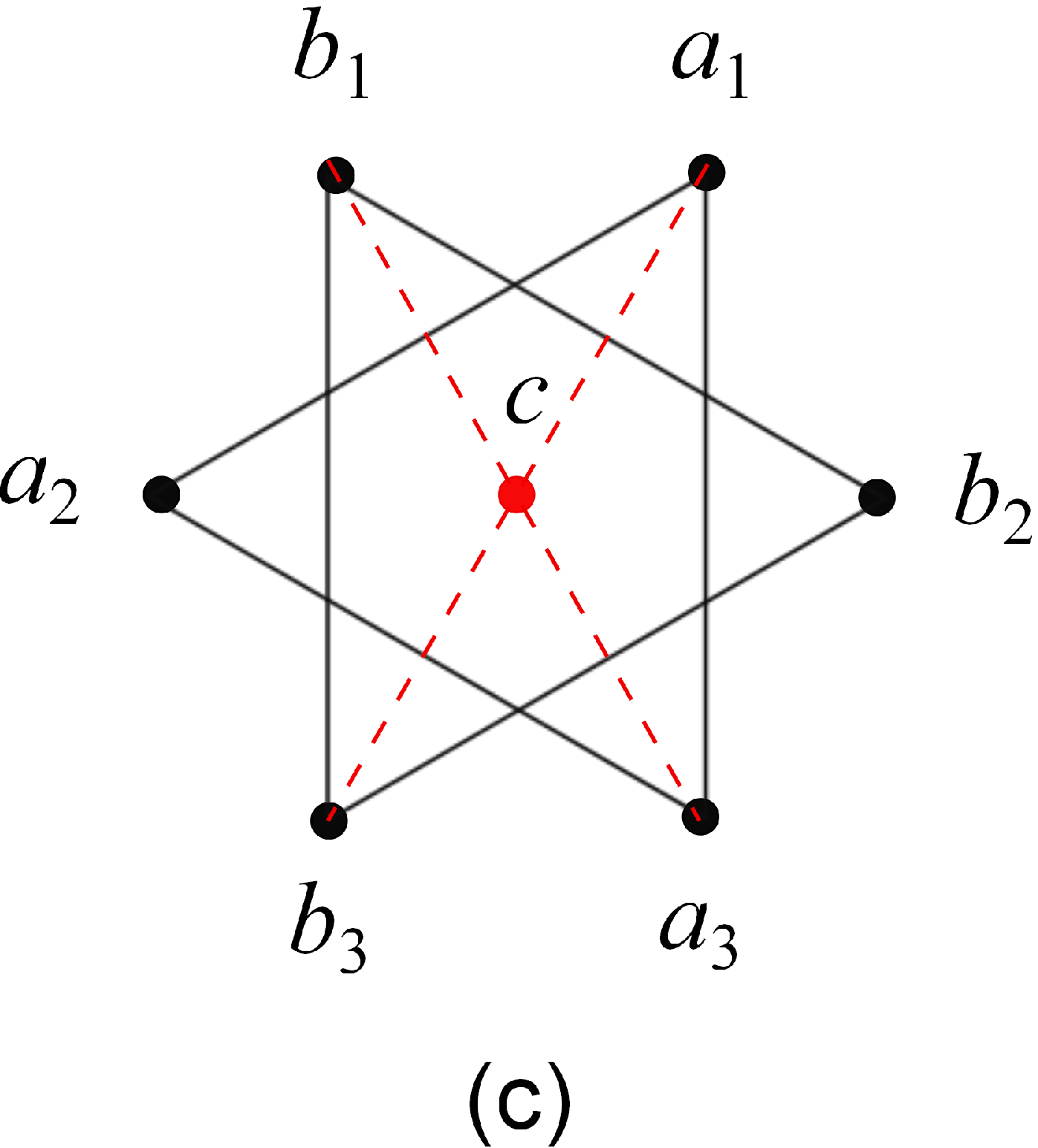} \hspace{.4cm}
\includegraphics[width=33mm]{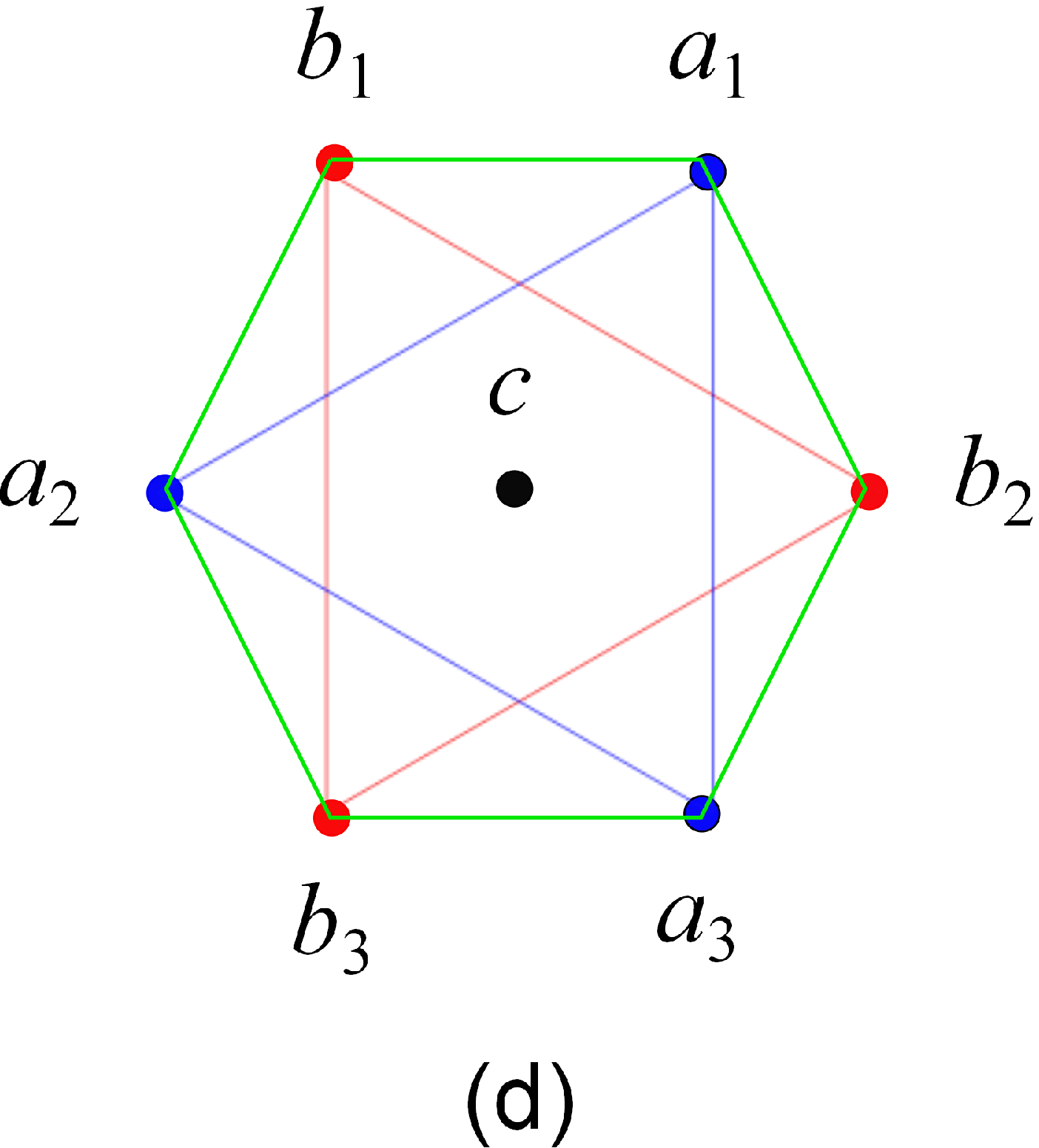}
\end{center}
\caption{ Star of David. } \label{starofDavidF}
\end{figure}

In Theorem \ref{gcdstarofdavid} part (2), we analyze whether $\gcd(a_1,a_2,a_3)=\gcd(b_1,b_2,b_3)$,
this is true if $\gcd(\rho, G_n(x)/\rho )=1$, where $\rho=\gcd(d(x),G_{1}(x))$.  The polynomials in Table \ref{equivalent}
that satisfy this condition are:
Fibonacci, Lucas, Pell-Lucas, Chebyshev first kind, Jacobsthal, Jacobsthal-Lucas, and both Morgan-Voyce polynomials.
The polynomials in Table \ref{equivalent} that satisfy that $\gcd(\rho^2, G_n(x))\not=1$ are:
Pell, Fermat, Fermat-Lucas, and Chebyshev second kind. We analyze these cases in Corollaries \ref{special:case:2} and \ref{special:case:3}.
For Theorem \ref{gcdstarofdavid} we use the points as given in Figure \ref{starofDavidF}
parts (a) and (b) with coordinates given in Tables \ref{coordinates:star:david(a)} and \ref{coordinates:star:david(b)}.

For simplicity we introduce the following notation that we use in the theorem and the corollaries below. We denote \label{delta:delta}
by $\Delta_a$ the set of vertices $\{ a_1,a_2, a_3\}$ and by $\Delta_b$ the set of vertices $\{b_1,b_2,b_3\}$ of
the two triangles of the star of David that are seen in Figure \ref{starofDavidF}. That is, the stars of David in the generalized Hosoya polynomial triangle.
For the rest of the paper we suppose that $a_2\not =  G_{0}(x) \, G_{0}(x)$.

\begin{theorem}\label{gcdstarofdavid} Suppose that $\Delta_a$ and $\Delta_b$ are as defined on page \pageref{delta:delta}. Let $c$ be the interior point of the star of David in the generalized Hosoya polynomial triangle. If $a_2\not =  G_{0}(x) \, G_{0}(x)$,  then
\begin{enumerate}[(1)]
\item  $a_1a_2a_3 = b_1b_2b_3$.
\item If  $m\ge 1$ and $n>1$, then
 \[
 \gcd(a_1,a_2,a_3)=
 \begin{cases}
         \beta \gcd(b_1,b_2,b_3), & \mbox{if $m$ and $n$ are both even;} \\
         \gcd(b_1,b_2,b_3), & \mbox{otherwise,}
\end{cases}
 \]
where $\beta$ is a constant that depends on $d(x), m$, and $n$.

\item If  $m\ge 0$ and $n\ge 0$, then
\[
 \gcd(a_1,a_2,a_3)=
 \begin{cases}
         \beta^{\prime} \gcd(b_1,b_2,b_3), & \mbox{if $m$ and $n$ are both odd;} \\
         \gcd(b_1,b_2,b_3), & \mbox{otherwise,}
\end{cases}
 \]
where $\beta^{\prime}$ is a constant that depends on $G_{1}(x), m$, and $n$.
\item The product $\gcd(a_1,b_3)\gcd(b_1,a_3)$ is equal to either $c$, $c G_{t}(x)$,  or  $c G_{t}^2(x)$, where $t=1$ if $G_{t}$ is Lucas type and  $t=2$ if $G_{t}$ is Fibonacci type.
\end{enumerate}
\end{theorem}

\begin{proof}
From the diagonal coordinates --of the star of David given in Figure \ref{starofDavidF}-- given
for $a_1,a_2, a_3$ and $b_1, b_2, b_3$ it is easy to see that part (1) is true.
We now observe that the star of David can be constructed in the Hosoya polynomial triangle if $m\ge 0$, $n\ge 2$.

We prove parts (2) and (3) together for the case in which the star of David is as in Figure \ref{starofDavidF} part (a).
The proof of the case of the star of David in Figure \ref{starofDavidF} part (b) is similar and we omit it.

From Lemma \ref{prop2;1} part (2) we have
\[\gcd(G_{m}(x)G_{n-1}(x),G_{m+1}(x)G_{n}(x))=\gcd(G_{m}(x),G_{n}(x))\gcd(G_{n-1}(x)G_{m+1}(x)).\]
Therefore,
\begin{eqnarray*}
	\gcd\left(b_1,b_3,b_2\right)&=&\gcd\left(\gcd\left(G_{m}(x)G_{n-1}(x),G_{m+1}(x)G_{n}(x)\right),G_{m+2}(x)G_{n-2}(x)\right)\\
							    &=&\gcd\left(\gcd(G_{m}(x),G_{n}(x))\gcd\left(G_{n-1}(x)G_{m+1}(x)\right) ,G_{m+2}(x)G_{n-2}(x)\right).
\end{eqnarray*}	
From Lemma \ref{gcddistance1;2} we know that
\[\gcd(G_{m+2}(x)G_{n-2}(x),\gcd(G_{n-1}(x),G_{m+1}(x)))=1.\]
So, 	
\begin{eqnarray}\label{triangle:bs}
	\gcd\left(b_1,b_3,b_2\right)&=&\gcd\left(\gcd\left(G_{m}(x),G_{n}(x)\right),G_{m+2}(x),G_{n-2}(x)\right)\nonumber\\
							    &=&\gcd\left(G_{m}(x),G_{n}(x),G_{m+2}(x)G_{n-2}(x)\right)\nonumber\\
							    &=&\gcd(G_{m}(x),\gcd(G_{n}(x),G_{m+2}(x)G_{n-2}(x))).
\end{eqnarray}
Let $D_t(x)=\gcd(G_{t}(x),G_{t-2}(x))$ for $t >0$. This, \eqref{triangle:bs}, and Lemma \ref{prop2;1} imply that
\begin{eqnarray}\label{trianglebb:cases}
	\gcd(b_1,b_3,b_2)
	&=&\gcd(G_{m}(x),\gcd(G_{n}(x),G_{m+2}(x)D_n(x)))\nonumber\\
	&=&\gcd(G_{n}(x),\gcd(G_{m}(x),G_{m+2}(x)D_n(x)))\nonumber\\
	&=&\gcd(G_{n}(x),\gcd(G_{m}(x),D_m(x)D_n(x)))\nonumber\\
	&=&\gcd(G_{n}(x),G_{m}(x),D_m(x)D_n(x)).
\end{eqnarray}
Similarly, we can see that
\[\gcd(a_1,a_2,a_3)=\gcd(G_{n-2}(x),G_{m+2}(x),D_n(x)D_m(x)).\]

We prove the remaining part of this proof by cases (GFP of Fibonacci type and Lucas type).

{\bf Case GFP of Fibonacci type}. Let's suppose that $G_n(x)=G^{\prime}_n(x)$ and we divide this case into three sub-cases depending on the parity of $m$ and $n$.

{\bf Sub-case $m$ and $n$ are odd}. From Lemma \ref{gcddistance1;2} part (5)  it easy to see that
$$\gcd(b_1,b_3,b_2)=\gcd(G^{\prime}_{n}(x),G^{\prime}_{m}(x),D_m(x)D_n(x))=1$$
and
$$\gcd(a_1,a_2,a_3)=\gcd(G^{\prime}_{n-2}(x),G^{\prime}_{m+2}(x),D_n(x)D_m(x))=1.$$

{\bf Sub-case $m$ and $n$ have different parity}. From Lemma \ref{gcddistance1;2} part (5) it is easy to see that $D_n(x)D_m(x)=G_{2}(x)$.
This and \eqref{trianglebb:cases} imply that $\gcd(b_1,b_3,b_2)=\gcd(G_{n}(x),G_{m}(x),G_{2}(x))$.
This and Lemma \ref{gcddistance1;2} part (1), imply that $\gcd(b_1,b_3,b_2)=1$.

{\bf Sub-case both $m$ and $n$ are even}. Suppose that $n=2k_1$ and $m=2k_2$. So, from Lemma \ref{gcddistance1;2} part (5) we have that
$D_m(x)=D_n(x)=d(x)$. Since  $G^{'}_0(x)=0$ and $G^{'}_1(x)=1$, by Proposition \ref{modulo:dx} we have
\begin{eqnarray*}
	G^{'}_{2k_1}(x)&\equiv&k_1g^{k_1-1}(x)d(x)\bmod{d^2(x)}\\
	G^{'}_{2k_2}(x)&\equiv&k_2g^{k_2-1}(x)d(x)\bmod{d^2(x)}.
\end{eqnarray*}
This and $\gcd(d(x),g(x))=1$ imply that
\[\gcd(b_1,b_2,b_3)=\gcd(k_1g^{k_1-1}(x)d(x),k_2g^{k_2-1}(x)d(x),d^2(x))=d(x)\gcd(d(x),k_1,k_2).\]
Similarly we have that
$\gcd(a_1,a_2,a_3)=d(x)\gcd(d(x),k_1-1,k_2+1).$

Let $\beta=\left(\gcd(d(x),k_1-1,k_2+1)\right)/\left(\gcd(d(x),k_1,k_2)\right)$. Therefore,
\[\gcd(a_1,a_2,a_3)=\beta\gcd(b_1,b_2,b_3).\]
Notice that if the star of David is as in Figure \ref{starofDavidF} part (b), then
\[\beta=\left(\gcd(d(x),k_1,k_2)\right)/\left(\gcd(d(x),k_1-1,k_2-1)\right).\]
This completes the proof of part (2).

We prove part (3) for the star of David in Figure \ref{starofDavidF} part (a).
The proof for the star of David in Figure \ref{starofDavidF} part (b) is similar and we omit it.

{\bf Case GFP of Lucas type.} Let's suppose that $G_n(x)=G^{*}_n(x)$. If $m$ and $n$ are not both even, then the proof follows in a similar way as seen above.

{\bf Sub-case both $m$ and $n$ are odd.} Suppose that  $n=2k_1+1$ and  $m=2k_2+1$. Therefore,  by Lemma \ref{gcddistance1;2} part (4) we know that $D_m(x)=D_n(x)=G^{*}_1(x)$. Since, $G^{*}_1(x)|d(x)$, by Proposition \ref{modulo:dx}
we have that
\begin{eqnarray*}
	G^{*}_{n}(x)&\equiv&n g^{k_1}(x)G^{*}_1(x) \bmod {(G^{*}_1(x))^2}\\
	G^{*}_{m}(x)&\equiv&m g^{k_2}(x)G^{*}_1(x) \bmod {(G^{*}_1(x))^2}.
\end{eqnarray*}
From this and \eqref{trianglebb:cases} is easy to see that
\[\gcd(b_1,b_2,b_3)=\gcd(n g^{k_1}(x)G^{*}_1(x),m g^{k_2}(x)G^{*}_1(x), {(G^{*}_1(x))^2}).\]
This and $\gcd(d(x),g(x))=1$ imply that
$ \gcd(b_1,b_2,b_3)=G^{*}_1(x)\gcd(n, m, G^{*}_1(x)).$

Similarly we can prove that
\[\gcd(a_1,a_2,a_3)=G^{*}_1(x)\gcd(G^{*}_1(x),n-2,m+2).\]
Let $\beta^{\prime}=\left(\gcd(G^{*}_1(x),n-2,m+2)\right)/\left(\gcd(G^{*}_1(x),n,m)\right)$. Then,
\[\gcd(a_1,a_2,a_3)= \beta^{\prime}\gcd(b_1,b_2,b_3).\]
Notice that if the star of David is as in Figure \ref{starofDavidF} part (b), then
\[\beta^{\prime}=\left(\gcd(G^{*}_1(x),n,m)\right)/\left(\gcd(G^{*}_1(x),n-2,m-2)\right).\]
This completes the proof of part (3).

We prove part (4) for the star of David in Figure \ref{starofDavidF} part (a).
The proof for the star of David in Figure \ref{starofDavidF} part (b) is similar and we omit it.
Using the diagonal coordinates we have that $\gcd(a_1,b_3)\gcd(b_1,a_3)$ is equal to
$$
\gcd(G_{m+1}(x) G_{n-2}(x),G_{m+1}(x)G_{n}(x))\gcd(G_{m}(x) G_{n-1}(x),G_{m+2}(x)G_{n-1}(x)).
$$
Therefore,
\[
\gcd(a_1,b_3)\gcd(b_1,a_3)= G_{m+1}(x) G_{n-1}(x) \gcd( G_{n-2}(x),G_{n}(x))\gcd(G_{m}(x) ,G_{m+2}(x)).
\]
 The conclusions follow using Proposition \ref{gcddistance1;2}.
\end{proof}

From the proof of Theorem \ref{gcdstarofdavid} it is easy to see the following corollaries.

\begin{corollary} \label{special:case:1} Suppose that $\Delta_a$ and $\Delta_b$ are as defined on page \pageref{delta:delta}, then
	 $\gcd(a_1,a_2,a_3)=\gcd(b_1,b_2,b_3)$, if $G_{t}(x)$ is one of the following polynomials: Fibonacci, Lucas,
	 Jacobsthal, Jacobsthal-Lucas,  Chebyshev first kind polynomials, Pell-Lucas, and both Morgan-Voyce polynomials.
\end{corollary}

\begin{corollary}\label{special:case:2} Suppose that $\Delta_a$ and $\Delta_b$ are as defined on page \pageref{delta:delta} and that
	 $\{G_{t}^{\prime}(x)\}$ is a GFP of Fibonacci type with $n=2k_1$ and $m=2k_2$.
	\begin{enumerate}
		\item If $G_{n}^{\prime}(x)$ is a Pell polynomial or Chebyshev polynomial of the second kind with  $k_1k_2\not\equiv 0 \bmod 4 $ and $k_1\not\equiv k_2 \bmod 2 $, then
				 $\gcd(a_1,a_2,a_3)=\gcd(b_1,b_2,b_3).$
				
			\item If $G_{n}^{\prime}(x)$ is a Fermat polynomial with $k_1k_2\not\equiv 0 \bmod 9 $ and $k_1\not\equiv 2 k_2 \bmod 3 $, then
		$\gcd(a_1,a_2,a_3)=\gcd(b_1,b_2,b_3).$	
		
		\item Suppose that $\Delta_a$ and $\Delta_b$ are as in Figure \ref{starofDavidF} part (a) and that $G_{n}^{\prime}(x)$ is a Fermat polynomial.
		\[ \text{If } k_1k_2\not\equiv 0 \bmod 9 \text{ and } k_1\not\equiv 2 k_2 \bmod 3, \text{ then } \gcd(a_1,a_2,a_3)=\gcd(b_1,b_2,b_3).\]	
				
		\item Suppose that $\Delta_a$ and $\Delta_b$ are as in Figure \ref{starofDavidF} part (b) and that $G_{n}^{\prime}(x)$ is a Fermat polynomial.
			\[ \text{If } k_1k_2\not\equiv 0 \bmod 9 \text{ and } (k_1-1)(k_2-1)\not\equiv 0 \bmod 9 , \text{ then } \gcd(a_1,a_2,a_3)=\gcd(b_1,b_2,b_3).\]		
				
	\end{enumerate}
\end{corollary}

 \begin{corollary}\label{special:case:3}  Suppose that $\Delta_a$ and $\Delta_b$ are as defined on page \pageref{delta:delta}. Let  $\{G_{t}^{\prime}(x)\}$
 	be the sequence of  Fermat-Lucas polynomials.
 				\begin{enumerate}
				 \item   Suppose that $\Delta_a$ and $\Delta_b$ are as in Figure \ref{starofDavidF} part (a).
				  \[ \text{If } nm\not\equiv 0 \bmod 9 \text{ and } (n-2)(m+2)\not\equiv 0 \bmod 9, \text{ then }
 				\gcd(a_1,a_2,a_3)=\gcd(b_1,b_2,b_3).\]
		
		 		\item   Suppose that $\Delta_a$ and $\Delta_b$ are as in Figure \ref{starofDavidF} part (b).
		 		 \[ \text{If } nm\not\equiv 0 \bmod 9 \text{ and } (n-2)(m-2)\not\equiv 0 \bmod 9 ,\text{ then }
 				\gcd(a_1,a_2,a_3)=\gcd(b_1,b_2,b_3).\]
		\end{enumerate}

 \end{corollary}

 \section{The star of David in the gibonomial triangle}

In this section we give a brief observation related to gibonomial coefficients.
Let $f^{*}_{k}(x)$ be the product of Fibonacci polynomials $F_{k}(x)F_{k-1}(x)\dots F_{1}(x)$.
Then the $n$th \emph{gibonomial coefficient} is defined by
$$\left[{n \brack r} \right] =\frac{f^{*}_{n}(x)}{f^{*}_{n-r}(x)f^{*}_{r}(x)}.$$
Notice that $f^{*}_{k}(1)$ gives rise to the classic (numerical) \emph{Fibonomial coefficient}
(see \cite{hoggattHillman}).  Koshy \cite{koshygibonomial} defines the  \emph{gibonomial triangle},
similarly as Pascal (binomial) triangle, where its entries are gibonomial coefficients instead of binomial
coefficients (see Table \ref{tabla3}). Note that Sagan and Savage \cite{Sagan} define \emph{lucanomials}, ${n \brace r}$,
 where $\left[{n \brack r} \right]$ is a particular case.

We now consider the star of David as in Figure \ref{starofDavidF} part (a) on page~\pageref{starofDavidF},
where the vertices are gibonomial coefficients (see Table \ref{coordinates:star:david_Gibonomial}).
Now it is easy to see that this star of David embeds in the  gibonomial triangle. Koshy \cite{koshygibonomial} proved  that $a_1a_2a_3 = b_1b_2b_3$.
In this section we establish the  second fundamental property of the star of David for the gibonomial triangle --the $GCD$ property--.
However, the property described in Figure \ref{starofDavidF} part (c) on page~\pageref{starofDavidF} does not hold in this triangle. Thus,
  $\gcd(a_1,b_3)\gcd(b_1,a_3)$ is not equal to $c=  \left[{n \brack r} \right]$.

\begin{table}[!ht]
	\begin{center}
		\begin{tabular}{llll}
			$a_{1}=\left[{n-1 \brack r} \right]$, \quad &$a_{2}=\left[{n \brack r-1} \right]$, & \text{ and } \quad& $a_{3}=\left[{n+1 \brack r+1} \right]$, \\ \\
			$b_{1}=\left[{n-1 \brack r-1} \right]$,         &$b_{2}=\left[{n \brack r+1} \right]$, & \text{ and }  & $b_{3}=\left[{n+1 \brack r} \right]$.
		\end{tabular}
	\end{center}
	\caption{Coordinates for the star of David in Figure \ref{starofDavidF} part (a) on page~\pageref{starofDavidF}.} \label{coordinates:star:david_Gibonomial}
\end{table}

\begin{table} [!ht] \footnotesize
  \begin{center} \addtolength{\tabcolsep}{-3pt} \scalebox{.9}{
	\begin{tabular}{ccccccccccc}
		&&&&&                                         $1$                                                             &&&&&\\
		&&&&                         $1$               &&         $1$                    			                       &&&&\\
		&&&                  $1$     &&               $x$         &&  	           $1$             	                   &&&\\
		&&            $1$     &&  $x^2+1$             &&        $x^2+1$           &&      $1$               		  &&\\
		&    $1$      &&  $x(x^2+2)$  &&      $2 + 3 x^2 + x^4$   &&          $x (x^2+2)$  &&            $1$          &\\
		$1$  &&$x^4+3 x^2+1$  &&  $x^6+5 x^4+7 x^2+2$ &&     $x^6+5 x^4+7 x^2+2$   &&     $x^4+3 x^2+1$  &&     $1$   \\
	\end{tabular}}
  \end{center}
\caption{The gibonomial triangle.} \label{tabla3}
\end{table}

\begin{theorem}\label{gcdstarofdavidforGibonomials} Let $S$ be the star of David as in Figure \ref{starofDavidF} part (a) on page~\pageref{starofDavidF}.
If $a_1, a_2, a_3, b_1, b_2$, and $b_3$ are the vertices of $S$ in the gibonomial triangle, then
	\begin{enumerate}[(1)]
		\item  $a_1a_2a_3 = b_1b_2b_3$.
		\item $\gcd(a_1,a_2,a_3)= \gcd(b_1,b_2,b_3)$.
	\end{enumerate}
\end{theorem}

\begin{proof} Koshy \cite{koshygibonomial} proved part (1) while Hillman and Hoggatt \cite{hoggattHillman} proved part (2) for Fibonomial coefficients (numerical).
The proof of part (2) is similar to the proof for Fibonomial coefficients.
\end{proof}

\section{Geometric interpretation of some identities of GFP}\label{further_results}

The aim of this section is to give geometrical interpretations of some polynomial identities that are known for
the Fibonacci numbers.  The novelty of this section is that we extend some well-known numerical identities to GFP  and provide
geometric proofs for these identities instead of the classical mathematical induction proofs.

Hosoya type triangles (polynomial and numeric) are good tools to discover, prove, or represent theorems geometrically.
Some properties that have been found and proved algebraically, are easy to understand when interpreted geometrically using  this triangle.
We now discuss some examples on how geometry of the triangle can be used to represent identities. The examples given in the following discussion are only
for the case in which the Hosoya polynomial triangle denoted by $H_{F}(x)$ has products of Fibonacci polynomials as entries. With this triangle in mind  we introduce a notation that will
be used in following examples.  We define an $n$-initial triangle $H_{F}^{n}(x)$ as the finite triangular arrangement formed by the first $n$-rows of the mentioned
Hosoya triangle with non-zero entries. Note that the initial triangle is the equilateral sub-triangle of the Hosoya triangle as in
Table \ref{tabla_equivalent} on page~\pageref{tabla_equivalent}  without the entries containing the factor $G_0$. For instance,
Table \ref{tabla2} on page~\pageref{tabla2} represents the $5$th initial triangle of $H_{F}(x)$.

If $F_{n}^{\prime}(x)$ represent the derivative of the  Fibonacci polynomial $F_{n}(x)$, then $F_{n}^{\prime}(x) =\sum_{k=1}^{n-1}F_{k}(x)F_{n-k}(x)$ (see \cite{koshy}).
The geometric  representation of this property in $H_{F}^{n}(x)$ is as follows: the derivative of the first point of the $n$th row of $H_{F}^{n}(x)$
 is equal to the sums of all points of the ($n-1$)th row of $H_{F}^{n}(x)$ (see Table \ref{tabla2} on page~\pageref{tabla2}). We have observed that this property
 implies that the integral of all points of $H_{F}^{n-1}(x)$  is equal to the sum of all points of one edge of the
 $H_{F}^{n}(x)$, where the constant of integration is $\lceil{n/2}\rceil$. This result is stated formally in Proposition \ref{integral}.

\begin{proposition}\label{integral}
Let $C$ be the constant of integration. Then
\begin{enumerate}
\item \begin{eqnarray*}
H(n,1)&=&\sum_{k=1}^{n-1} \int {H(n-1,k)}.
\end{eqnarray*}
Equivalently,
 \begin{eqnarray*}
F_{n}(x)&=&\sum_{k=1}^{n-1} \int {F_{k}(x)F_{n-k}(x)},
\end{eqnarray*}
where $C=1$ if $n$ is odd and zero otherwise.

\item
\begin{eqnarray*}
	H(n+1,1)+H(n,1)-1		&=&x \sum_{r=1}^{n} \sum_{k=1}^{r-1} \int {H(r-1,k)}.
	\end{eqnarray*}
Equivalently,
 \begin{eqnarray*}
	F_{n+1}(x)+F_{n}(x)-1	&=&x\sum_{r=1}^{n} \sum_{k=1}^{r-1} \int {F_{k}(x)F_{r-k}(x)},
\end{eqnarray*}
where $C=\lceil{n/2}\rceil$.

\end{enumerate}
\end{proposition}

\begin{proof}
The proof of part (1) is straightforward using the geometric interpretation of $F_{n}^{\prime}(x)$.

We prove part (2). From part (1) and from the geometry of $H_{F}^{n-1}(x)$ it is easy to see that
$\sum_{r=1}^{n} H(k,1)=\sum_{r=1}^{n-1} \sum_{k=1}^{r-1} \int {H(r-1,k)}$.
From Koshy \cite[Theorem 37.1]{koshy} we know that
$F_{n+1}(x)+F_{n}(x)-1=x\sum_{i=1}^{n} F_{i}(x)$.
This and the fact that $H(t,1)=F_{t}(x)$ for all $t\ge 1$ completes the proof.
\end{proof}

\begin{lemma}\label{Propiedad:de:las:paralelas}
If $i$, $j$, $k$, and $r$ are nonnegative integers with $k+j\le r$, then in the Hosoya polynomial triangle it holds that
\[
	H(r+2i,k+j+i)-H(r+2i,k+i)=(-1)^i\gamma(x)(H(r,k+j)-H(r,k)).
\]
\end{lemma}

The proof of the Lemma \ref{Propiedad:de:las:paralelas} follows using induction
and the rectangle property which states that $H(n,m)=\delta(x)H(n-1,m)+\gamma(x)H(n-2,m)$ (see Figure \ref{rectangle}).

\begin{figure} [!ht]
\begin{center}
\includegraphics[width=125mm]{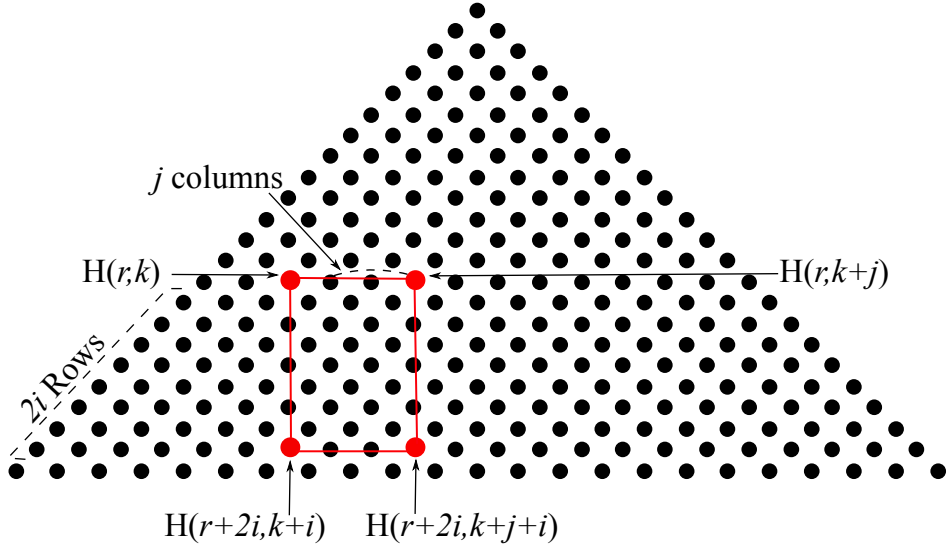}
\end{center}
\caption{Property of Rectangle.} \label{rectangle}
\end{figure}

It is well known that the Catalan identity is a generalization of the Cassini identity. In Wolfram MathWord there is another numerical generalization of the Cassini and Catalan identities, called the Johnson identity \cite{WolframMathWorld}. It states that for the Fibonacci number sequence $\{F_{n}\}$,
$$F_{a}F_{b}-F_{c}F_{d}=(-1)^{r}\left(F_{a-r}F_{b-r}-F_{c-r}F_{d-r}\right)$$
 where $a, b, c, d$, and $r$ are arbitrary integers with $a+b=c+d$.

The example in Figure \ref{Cassini_Catalan} gives a geometric representation of the numeric identities
(the same representation holds for polynomials). To represent the Cassini identity we take two consecutive points in the Hosoya triangle along a horizontal line such that one point is located in the central column of the triangle, see Figure \ref{Cassini_Catalan}. We then pick two other arbitrary
consecutive points $P_1$ and $P_2$ such that they form a vertical rectangle along with the first pair of points. Now it easy to see that subtracting the horizontal points $P_1$ and $P_2$ gives $\pm 1$. Since the entries of the triangle are products of
Fibonacci numbers, we obtain the Cassini identity.

The second example in Figure \ref{Cassini_Catalan} represent the Catalan identity. In this case we take any two horizontal points
$Q_1$ and $Q_2$ where $Q_1$ is located (arbitrarily) in the central column of the triangle. We then pick other two arbitrary
points $P_1$ and $P_2$ such that those form a rectangle with $Q_1$ and $Q_2$. Now it easy to see that subtracting
the horizontal points $P_1$ and $P_2$ gives $\pm (Q_{1}- Q_{2})$. Since the entries of the triangle are products of
Fibonacci numbers, we obtain the Catalan identity. Note that if we eliminate the condition that $Q_1$
must be in the central column, we obtain the Johnson identity.

Theorem \ref{johnson} is the generalization of the Johnson identity.  As a consequence of Theorem \ref{johnson} we state Corollary \ref{catalan:casini} --this generalizes Catalan and Cassini identities for GFP--.

\begin{figure} [!ht]
\begin{center}
\includegraphics[width=150mm]{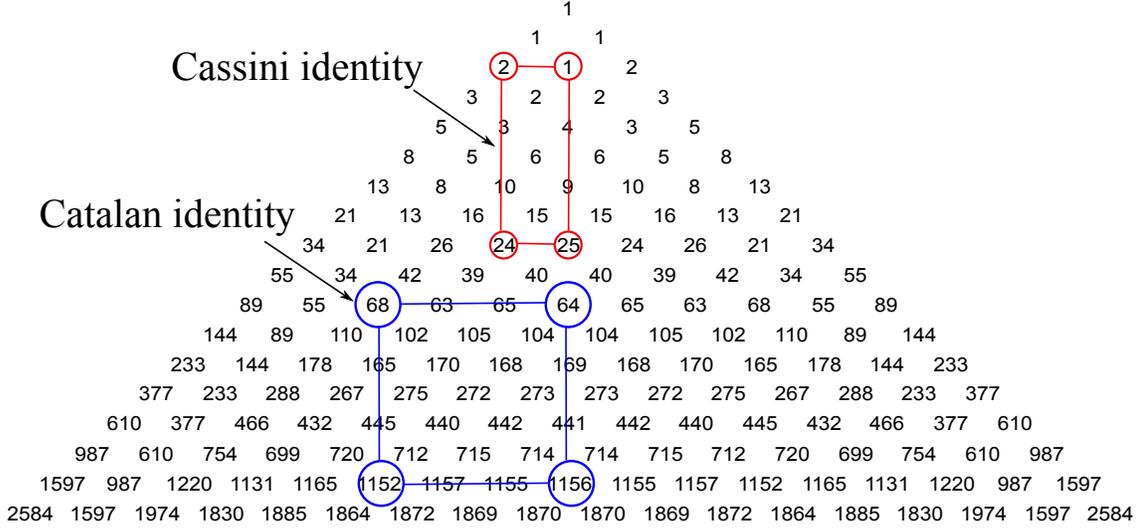}
\end{center}
\caption{Cassini and Catalan identities.} \label{Cassini_Catalan}
\end{figure}

\begin{theorem}\label{johnson}
Let $a, b, c, d$ and $t$ be nonnegative integers with $\min\{a,b,c,d\}-t$ non-negative. If $\{G_n(x)\}$ is GFP and $a+b=c+d$ then
\[
	\begin{vmatrix}
	G_{a}(x) & G_{c}(x) \\
	G_{d}(x) & G_{b}(x) \end{vmatrix}
	= (-1)^t g^t(x)
	\begin{vmatrix} G_{a-t}(x) & G_{c-t}(x) \\
	G_{d-t}(x) & G_{b-t}(x) \end{vmatrix}.
	\]
\end{theorem}
\begin{proof}
Let $i$, $j$, $k$, and $r$ be nonnegative integers such that $a=k+j+i$, $b=r+i-k-j$, $c=k+i$, $d=r+i-k$, and $t=i$. Therefore, by
Lemma (\ref{Propiedad:de:las:paralelas}) and  Proposition (\ref{lemma0}) the equality holds.
\end{proof}

\begin{corollary}\label{catalan:casini}
	Suppose that $m, r$ are non-negative integers. If $\{G_n(x)\}$ is GFP, then
	
		\begin{enumerate}
		\item (Catalan identity)  	
		\[
		\left|\begin{array}{ll}
		G_{m}(x) & G_{m+r}(x) \\
		G_{m-r}(x) & G_{m}(x)
		\end{array}\right|
		=(-1)^{m-r} g^{m-r}(x)
		\left|\begin{array}{ll} G_{r}(x) & G_{2r}(x) \\
		G_{0}(x) & G_{r}(x)
		\end{array}\right|,
		\]	
		
		\item (Cassini identity)
		\[
		\left|\begin{array}{ll}
		G_{m}(x) & G_{m+1}(x) \\
		G_{m-1}(x) & G_{m}(x)
		\end{array}\right|
		=(-1)^{m-1} g(x)^{m-1}
			\left|\begin{array}{ll}
		G_{1}(x) & G_{2}(x) \\
		G_{0}(x) & G_{1}(x)
		\end{array}\right|.
		\]
	\end{enumerate}
\end{corollary}

\begin{proof}  The proof is straightforward when the appropriate values of $m$ and $r$ are substituted in Theorem \ref{johnson}
(see Figure \ref{Cassini_Catalan}). If we evaluate both determinants in Theorem \ref{johnson} we obtain four summands that are four points
in the Hosoya polynomial triangle. Note that these four points are the vertices of a rectangle in the Hosoya triangle.
\end{proof}

For the next result we introduce the function
 \[
 I(n)=
 \begin{cases}
         g(x), & \mbox{if $n$ is even;} \\
         1, & \mbox{if $n$ is odd.}
\end{cases}
 \]

 \begin{figure} [!ht]
\begin{center}
\includegraphics[width=100mm]{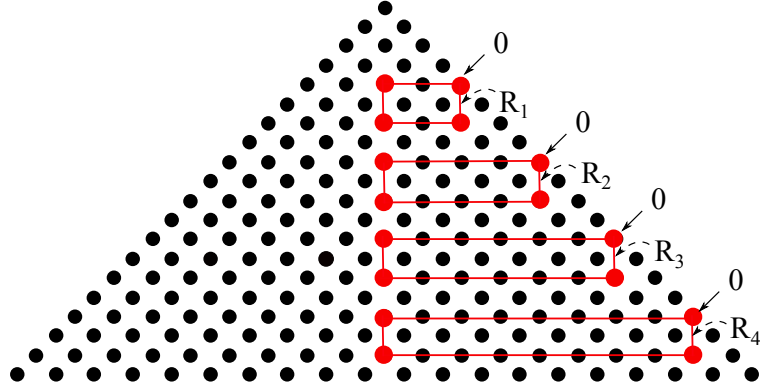}
\end{center}
\caption{Geometrical interpretation of Theorem \ref{sums:property}. } \label{rectangleTheorem12}
\end{figure}

We observe that if we have a Hosoya triangle where the entries are products of GFP of Fibonacci type, then we can draw rectangles
with two vertices in the central line (perpendicular bisector) of the triangle and a third vertex on the edge of the triangle
(see Figure \ref{rectangleTheorem12}). Let $R$ be a rectangle with the extra condition that the upper vertex points are multiplied by $g(x)$,
then Lemma \ref{Propiedad:de:las:paralelas} guarantees that the sum of the two top vertices of $R$  is equal to  the sum of the remaining vertices of $R$.
Since the points in the edge of this triangle are equal to zero, we have that one of the vertices of $R$ is equal to zero. The other vertex in the same vertical
line is a GFP multiplied by one. This geometry gives rise to Theorem \ref{sums:property}.

\begin{theorem}\label{sums:property} Suppose that $n$ and $k$ are positive integers. If $\{G_n^{'}(x)\}$ is of Fibonacci type  then
	$$
	\sum_{j=2}^{2n+1} I(j)G_{j}^{'2}(x) =\sum_{j=1}^n G_{4j+1}^{'}(x)
	$$
	and
	$$
	\sum_{j=2}^{2n+1} (-1)^{j+1}I^2(j)G_{2j}^{'2} (x)=  d(x)\sum_{j=1}^{n}G_{8j+2}^{'}(x) .
	$$
\end{theorem}

\begin{proof} First of all we recall that $G_{1}^{'}(x)=1$. We prove the first identity.
	\begin{eqnarray*}
		\sum_{j=2}^{2n+1} I(j)G_{j}^{'2} (x)&=&\sum_{j=1}^{n} (G_{2j+1}^{'2}(x)+ g(x)G_{2j}^{'2} (x))
		=\sum_{j=1}^n (G_{4j+1}^{'}(x)G_1^{'}(x)+G_{0}^{'}(x)G_{4j}^{'})\\
		&=&\sum_{j=1}^n G_{4j+1}^{'}(x).
	\end{eqnarray*}
	
	We now prove the second identity.  Let  $S:=\sum_{j=2}^{2n+1} (-1)^{j+1}I^2(j)G_{2j}^{'2} (x)$. Lemma \ref{Propiedad:de:las:paralelas}  implies that
	\begin{eqnarray*}
		S	&=&\sum_{j=1}^{n} (G_{4j+2}^{'2}(x)- g^2(x)G_{4j}^{'2} (x))\\
		&=&\sum_{j=1}^{n} \Big(\Big(G_{4j+2}^{'2}(x)+g(x)G_{4j+1}^{'2}(x)\Big)-g(x)\Big(G_{4j+1}^{'2}(x) +g(x)G_{4j}^{'2} (x)\Big)\Big).\\
	\end{eqnarray*}	
Since $G_{1}^{'}(x)=1$, we have
\begin{eqnarray*}
		S	&=&\sum_{j=1}^{n} \left(\left(G_{8j+3}^{'}(x)+g(x)G_{8j+1}^{'}(x)G_{0}^{'}(x)\Big)-g(x)\Big(G_{8j+1}^{'}(x)+g(x)G_{8j-1}^{'}(x)G_{0}^{'}(x)\right)\right)\\
		&=& G_{1}^{'}(x)\sum_{j=1}^{n} \Big(G_{8j+3}^{'}(x)-g(x)G_{8j+1}^{'}(x)\Big)=d(x)\sum_{j=1}^{n}G_{8j+2}^{'}(x) .
	\end{eqnarray*}
	This proves the theorem.
\end{proof}

Corollary \ref{Corollay:sums:property} provides a closed formula for special cases of Theorem \ref{sums:property}.
We use Figure \ref{rectangleCorollary14} to have a geometric interpretation of Corollary  \ref{Corollay:sums:property}.
For simplicity we only prove part (2) (part (1) is similar and we omit it). From the geometric point of view
Corollary  \ref{Corollay:sums:property} part (2) states that the sum of all points that are in the intersection of a
finite zigzag configuration and the central line of the triangle is the last
point of the zigzag configuration (see Figure \ref{rectangleCorollary14}). We now give more details of the validity of this statement.
 From the hypothesis of Corollary  \ref{Corollay:sums:property} we have that
 $g=1$ and  $H(0,k)=H(k,0)=0$ for every $k$. This and the definition of the Hosoya polynomial sequence, on page~\pageref{HosoyaSection}, imply that
$$H(r,k)= d(x) H(r-1,k)+ H(r-2,k)    \text{ and  } H(r,k)= d(x) H(r-1,k-1)+H(r-2,k-2).$$
Therefore the points depicted in Figure \ref{rectangleCorollary14} have the properties described in Table \ref{PointZigZag}.
\begin{table}[!ht]
	\begin{center}
		\begin{tabular}{llll}
			$p_{0}=0$,                    		&$p_{2}=d(x)p_{1}+p_{0}$,        		& $p_{4}=d(x)p_{3}+p_{2}$ & $p_{6}=d(x)p_{5}+p_{4}$ \\
			$p_{6}=d(x)p_{5}+p_{4}$,\quad \quad & $p_{8}=d(x)p_{7}+p_{6}$, \quad\quad&  \hspace{1cm}\dots  & $p_{4n}=d(x)p_{4n-1}+p_{4n-2}$.
		\end{tabular}
	\end{center}
	\caption{Properties of points in the Zigzag Figure \ref{rectangleCorollary14}.} \label{PointZigZag}
\end{table}

Since $g=1$, we have that $I(j)=1$ for all $j$. Therefore, $\sum_{j=1}^{2n+1} I(j)G_{j}^{'2}(x)$ is actually the sum of all points that are
in the intersection of the  zigzag diagram with central line of the triangle (see Figure \ref{rectangleCorollary14}). Thus,
$$d(x)\sum_{j=1}^{2n+1} G_{j}^{'2}(x) =p_{0}+d(x)p_{1}+d(x)p_{3}+d(x)p_{5}+d(x)p_{7}+\dots+d(x)p_{4n-1}.$$
 The first two terms of the right side of this sum are equal to
the third point of the  zigzag diagram (see Table \ref{PointZigZag} and Figure \ref{rectangleCorollary14}). Therefore, substituting them by $p_{2}$ we have
$$d(x)\sum_{j=1}^{2n+1} G_{j}^{'2}(x) =p_{2}+d(x)p_{3}+d(x)p_{5}+d(x)p_{7}+\dots+d(x)p_{4n-1}.$$
 The first two terms of the right side of this sums are equal to
the the fifth point of the  zigzag diagram (see Table \ref{PointZigZag} and Figure \ref{rectangleCorollary14}). Therefore, substituting them by $p_{4}$ we have
 $$d(x)\sum_{j=1}^{2n+1} G_{j}^{'2}(x) =p_{4}+d(x)p_{5}+d(x)p_{7}+\dots+d(x)p_{4n-1}.$$
 Similarly, we substitute $p_{4}+d(x)p_{5}$ by the seventh point of the  zigzag diagram. Thus,
 $$d(x)\sum_{j=1}^{2n+1} G_{j}^{'2}(x) =p_{6}+d(x)p_{7}+\dots+d(x)p_{4n-1}.$$
 We keep  systematically doing those substitutions to obtain
$$d(x)\sum_{j=1}^{2n+1} G_{j}^{'2}(x) =p_{4n}=G_{2n-1}^{'}(x)G_{2n}^{'}(x).$$ This completes the geometric proof of
Corollary \ref{Corollay:sums:property} part (2).

\begin{figure} [!ht]
	\begin{center}
		\includegraphics[width=120mm]{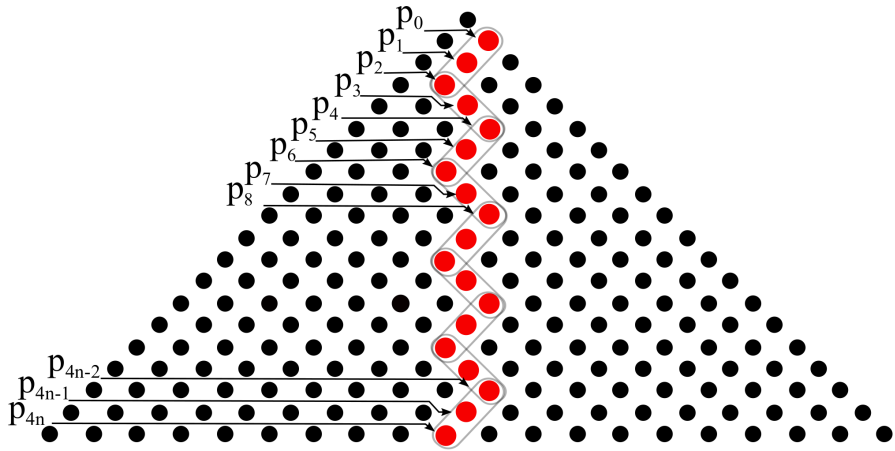}
	\end{center}
	\caption{Geometric interpretation of Corollary \ref{Corollay:sums:property}. } \label{rectangleCorollary14}
\end{figure}

\begin{corollary}\label{Corollay:sums:property}
	If $\{G_n^{'}(x)\}$ is of Fibonacci type then,
	\begin{enumerate}
		\item 	\[\sum_{j=1}^{n}g^{2(n-j)}G_{4j-3}^{'}(x)=\frac{G_{2n-1}^{'}(x)G_{2n}^{'}(x)}{d(x)},\]
		\item  if  $\{G_n^{'}(x)\}$ satisfies that  $g=1$, then
		\[	\sum_{j=1}^{2n-1} I(j)G_{j}^{'2}(x)=\frac{G_{2n-1}^{'}(x)G_{2n}^{'}(x)}{d(x)}.\]
	\end{enumerate}	
\end{corollary}

\begin{proof} Since $H(2n,n)=G_{n}^2(x)$, we have that   $g^n(x)H(1,1)+\sum_{j=1}^{n} d(x)g^{n-j}G_{j}^2(x)$ equals
	
\begin{eqnarray*}
	\sum_{j=1}^{n} d(x)g^{n-j}H(2j,j)
	&=&g^{n-1}(x)(g(x)H(1,1)+d(x)H(2,1))+\sum_{j=2}^{n} d(x)g^{n-j}G_{j}^2(x)\\
	&=&g^{n-1}(x)H(3,1)+d(x)g^{n-2}H(4,2)+\sum_{j=3}^{n} d(x)g^{n-j}G_{j}^2(x)\\
	&=&g^{n-2}(x)H(5,3)+\sum_{j=3}^{n} d(x)g^{n-j}G_{j}^2(x)\\
	&=&g^{n-2}(x)G_{3}(x)G_{2}(x)+\sum_{j=3}^{n} d(x)g^{n-j}G_{j}^2(x).
\end{eqnarray*}
Similarly we obtain that
\begin{equation}\label{general}
\sum_{j=1}^{n} d(x)g^{n-j}G_{j}^2(x)=H(2n+1,n+1)=G_{n+1}(x)G_{n}(x)-g^n(x)G_{1}(x)G_{0}(x).
\end{equation}
Note that
\begin{eqnarray*}
	\sum_{j=2}^{2n+1} g^{2n+1-j}(x)G_{j}^{'2} (x)&=&\sum_{j=1}^{n} g^{2n-j}(x)(G_{2j+1}^{'2}(x)+ g(x)G_{2j}^{'2} (x))\\
	&=&\sum_{j=1}^n g^{2n-j}(x)(G_{4j+1}^{'}(x)G_1^{'}(x)+G_{0}^{'}(x)G_{4j}^{'})\\
	&=&\sum_{j=1}^n g^{2n-j}(x)G_{4j+1}^{'}(x).
\end{eqnarray*}
This, the equation \eqref{general}, and $G_0(x)=0$ completes the proof.
\end{proof}

\section{Numerical types of Hosoya triangle}

In this section we study some connections of the Hosoya polynomial triangles with some numeric sequences that maybe found in \cite{sloane}.
We show that when we evaluate the entries of a Hosoya polynomial triangle at $x=1$ they give a triangle that is in \url{http://oeis.org/}.
The first Hosoya triangle is the classic Hosoya triangle formerly  called Fibonacci triangle.

We now introduce some notation that is used in Table \ref{Tableahosoyatriangles}. Let $H_{F}(x)$
denote the Hosoya polynomial triangle with products of Fibonacci polynomials as entries. Similarly
we define the notation for the Hosoya polynomial triangle of the other types
--Chebyshev polynomials, Morgan-Voyce polynomials, Lucas polynomials, Pell polynomials,
Fermat polynomials, Jacobsthal polynomials--. The star of David property holds obviously for all these
numeric triangles.

\begin{table}[!ht]
\begin{center}
\begin{tabular}{|l|l|l|l|l|l|l|l|}
\hline
Type triangle 		& Notation 		&  Entries 					& Sloane 			  	\\ \hline
Fibonacci 	    	& $H_F(1)$		& $F_k(1)F_{r-k}(1)$      		&  \seqnum{A058071} 	\\
Lucas 	    		&$H_D(1)$		& $D_k(1)D_{r-k}(1)$      		&\seqnum{A284115} 	\\
Pell		    	& $H_P(1)$		& $P_k(1)P_{r-k}(1)$      		&  \seqnum{A284127} 	\\
Pell-Lucas    		&$H_Q(1)$		& $Q_k(1)Q_{r-k}(1)$      		&  \seqnum{A284126}	\\
Fermat 	    		&$H_{\Phi}(1)$	& $\Phi_k(1)\Phi_{r-k}(1)$		& \seqnum{A143088} 	\\
Fermat-Lucas 		&$H_{\vartheta}(1)$	& $\vartheta_k(1)\vartheta_{r-k}(1)$ & \seqnum{A284128} \\
Jacobsthal   		&$H_J(1)$		& $J_k(1)J_{r-k}(1)$      		&\seqnum{A284130} 	\\
Jacobsthal-Lucas 	& $H_j(1)$		& $j_k(1)j_{r-k}(1)$      		& \seqnum{A284129} 	\\
Morgan-Voyce		& $H_B(1)$		& $B_k(1)B_{r-k}(1)$      		&\seqnum{A284131}		\\
Morgan-Voyce 		&$H_C(1)$		& $C_k(1)C_{r-k}(x)$      		&\seqnum{A141678}		\\
\hline
\end{tabular}
\end{center}
\caption{Numerical Hosoya triangles present in Sloane \cite{sloane}.} \label{Tableahosoyatriangles}
\end{table}

We also observe some curious numerical patterns when we compute the GCD's of the coefficients of polynomials discussed in this paper.
In particular, the GCD of the coefficients of $\Phi_n(x)$ --the $n$th Fermat polynomial-- is $3^{a_{n}}$ where $a_{n}$
is the $n$th element of \seqnum{A168570}. The GCD of the coefficients of $\vartheta_n(x)$ --the $n$th
Fermat-Lucas polynomial-- is $3^{a_{n}}$ where $a_{n}$ is the $n$th
element of \seqnum{A284413}. We also found that the GCD of the coefficients of the $P_{2n}(x)$ --the $2n$th
Pell polynomial-- is $2^{a_{n}}$ where $a_{n}$ is the $n$th element of \seqnum{A001511}.
Finally, the GCD of the coefficients of the $U_{n}(x)$ --the $n$th Chebyshev's polynomial of second kind--
is $2^{a_{n}}$ where $a_{n}$ is the $n$th element of \seqnum{A007814}.

\section{Acknowledgement}

The first and last authors were partially supported by The Citadel Foundation.

\bigskip
\hrule
\bigskip

\noindent 2010 {\it Mathematics Subject Classification}:
Primary 11B39; Secondary 11B83.

\noindent \emph{Keywords: }
Hosoya triangle, Gibonomial triangle, Fibonacci polynomial, Chebyshev polynomials,
Morgan-Voyce polynomials, Lucas polynomials, Pell polynomials, Fermat polynomials.

\end{document}